\documentclass[a4paper,12pt]{article}
\usepackage[centertags]{amsmath}
\usepackage{amsfonts}
\usepackage{amssymb}
\usepackage{amsthm}
\usepackage{amsmath}
\usepackage{dsfont}
\usepackage{graphicx}
\usepackage{tikz, subfigure}
\usepackage{pstricks-add}
\usepackage{verbatim}

\usepackage[latin1]{inputenc}
\usepackage{tikz}
\definecolor {processblue}{cmyk}{0.96,0,0,0}
\usepackage{amsfonts,graphicx,amsmath,amssymb,hyperref,color}
\usepackage[english]{babel}
\usepackage{authblk}

\newtheorem{Theorem}{Theorem}[section]
\newtheorem{Proposition}{Proposition}[section]
\newtheorem{Lemma}{Lemma}[section]
\newtheorem{Corollary}{Corollary}[section]

\newtheorem{Definition}{Definition}[section]

\newtheorem{Example}{Example}[section]

\newtheorem{Remark}{Remark}[section]

\newcounter{tmp}

\begin{document}
\title{xxxx}
\date{\today}
 \title{On univoque points for self-similar sets}
\author{Simon Baker\thanks{simonbaker412@gmail.com}}
\author{Karma Dajani\thanks{k.dajani1@uu.nl}}
\author{Kan Jiang\thanks{K.Jiang1@uu.nl}}
\affil{Department of Mathematics, Utrecht University}
\maketitle
\begin{abstract}
Let $K\subseteq\mathbb{R}$ be the unique attractor of an iterated
function system. We consider the case where
$K$ is an interval and study those elements of $K$ with a unique
coding. We prove under mild conditions that the set of points with a
unique coding can be identified with a subshift of finite type. As a
consequence of this, we can show that the set of points with a unique
coding is a graph-directed self-similar set in the sense of Mauldin
and Williams \cite{MW}. The theory of Mauldin and Williams then
provides a method by which we can explicitly calculate the Hausdorff
dimension of this set. Our algorithm can be applied generically, and
our result generalises the work of \cite{DKK}, \cite{K1}, \cite{K2},
and \cite{MK}.

\end{abstract}

\section{Introduction}

Let $\{f_{j}\}_{j=1}^{m}$ be an iterated function system (IFS) of similitudes which are defined on $\mathbb{R}$ by
\[f_j(x)=r_jx+a_j.\]
Where the similarity ratios satisfy $0<r_j<1$ and the translation parameter $a_j\in \mathbb{R}$. It is well known that there exists a unique non-empty compact set $K\subset \mathbb{R}$ such that
\begin{equation}
\label{Hutchinson formula}
K=\bigcup_{j=1}^{m}f_j(K).
\end{equation}
We call $K$ the self-similar set or attractor for the IFS
$\{f_{j}\}_{j=1}^{m}$, see \cite{H} for further details. We refer to the elements of $\{f_{j}(K)\}^{m}_{j=1}$ as
first-level intervals when $K$ is an interval. An IFS is
called homogeneous if all the similarity ratios $r_j$ are  equal.
For any  $x \in K$, there exists a sequence
$(i_n)_{n=1}^{\infty}\in\{1,\ldots,m\}^{\mathbb{N}}$ such that
\[x=\lim_{n\to \infty}f_{i_1}\circ \cdots\circ f_{i_n}(0)=\bigcap_{n=1}^{\infty}f_{i_1}\circ \cdots\circ f_{i_n}(K).\] We call such a sequence a coding of $x$. The attractor $K$ defined by (\ref{Hutchinson formula}) may equivalently be defined to be the set
of points in $\mathbb{R}$ which admit a coding, i.e.,
we can define a surjective projection map between the symbolic space $\{1,\ldots, m\}^{\mathbb{N}}$ and the self-similar set $K$ by
 $$\pi((i_n)_{n=1}^{\infty}):=\lim_{n\to \infty}f_{i_1} \circ\cdots \circ
f_{i_n}(0).$$
An  $x\in K$ may
have many different codings, if $(i_n)_{n=1}^{\infty}$ is unique then we call $x$ a
univoque point. The set of univoque points is called the univoque
set and we denote it by $U_{\{f_j\}^{m}_{j=1}},$ i.e.,
\begin{align*}
U_{\{f_j\}^{m}_{j=1}}:=\Big\{x\in K: &\textrm{ there exists a unique } (i_{n})_{n=1}^{\infty}\in \{1,\ldots,m\}^{\mathbb{N}} \textrm{ satisfying }\\
& x=\lim_{n\to \infty}f_{i_1}\circ \cdots\circ f_{i_n}(0)\Big\}.
\end{align*}
Let $\widetilde{U}_{\{f_j\}^{m}_{j=1}}:=\pi^{-1}(U_{\{f_j\}^{m}_{j=1}}).$
If there is no risk of confusion, we denote ${U}_{\{f_j\}^{m}_{j=1}}$ and $\widetilde{U}_{\{f_j\}^{m}_{j=1}}$ by $U$ and $\widetilde{U}$ respectively. With a little effort, it may be shown that $\pi$ is a homeomorphism between   the set of unique
codings $\widetilde{U}$ and
 the univoque set $U$. In this paper we present a
general algorithm for determining the Hausdorff dimension of
$U$ when $K$ is an interval. Unless stated
otherwise, in what follows we will always assume that our IFS is
such that $K$ is an interval.

 Part of our motivation comes
from the study of $\beta$-expansions. Given $\beta>1$ and
$x\in\left[0,(\lceil \beta \rceil-1)(\beta-1)^{-1}\right]$ there
exists
 a sequence $(a_{n})_{n=1}^{\infty}\in \{0,\ldots, \lceil \beta \rceil-1\}^{\mathbb{N}}$ such that $$x=\sum_{n=1}^{\infty}a_{n}\beta^{-n}.$$
We call such a sequence a $\beta$-expansion of $x$.
 Expansions in non-integer bases were pioneered in the papers of Renyi \cite{R} and Parry \cite{P}.
 For more information, see \cite{EHJ}, \cite{KarmaCor}, \cite{MK} and the references therein.

We can study $\beta$-expansions via the IFS
\[g_j(x)=\dfrac{x+j}{\beta},\,j\in\{0,\ldots,\lceil \beta \rceil-1\}.\]
The self-similar set for this IFS is the interval
$\mathcal{A_\beta}:=\left[0,(\lceil \beta \rceil-1)(\beta-1)^{-1}\right]$. For
$\beta$-expansions, it is clear that any first-level interval
$g_{j}(\mathcal{A_\beta})$ intersects at most two other first-level intervals
simultaneously. For any $M\in \mathbb{N}$,  it is  straightforward to show that
\[g_{i_1}\circ\cdots \circ
g_{i_M}(0)=\sum_{n=1}^{M}i_{n}\beta^{-n}.\]
Therefore, $\lim_{n\to \infty}g_{i_1}\circ g_{i_2} \circ\ldots \circ
g_{i_n}(0)=x$ if and only if $(i_n)_{n=1}^{\infty}$ is a $\beta$-expansion of $x$.  Much work has been done on
the set of points with a unique $\beta$-expansion. Glendinning and Sidorov  classified in \cite{GS}  those
$\beta\in(1,2)$ for which the Hausdorff dimension of the univoque set is positive. However, their approach did not allow them to
calculate the Hausdorff dimension. This result was later generalised
to arbitrary $\beta>1$ in \cite{KLD}. Dar\'{o}czy and K\'{a}tai
\cite{DKK} offered an approach to the problem of calculating the
dimension when $\beta\in(1,2)$, but they could only calculate the
dimension when $\beta$ is a special purely Parry number \cite{P}. The reason why they chose special numbers is that the directed graph they constructed was strongly connected, we however will prove that this is not necessary.
Making use of similar ideas, Kall\'{o}s \cite{K1}, \cite{K2} showed
that for $\beta>2$:
\begin{itemize}
\item[(1)]  If $\beta\in[\lceil \beta \rceil-1,(\lceil \beta
\rceil-1+\sqrt{(\lceil \beta \rceil)^2-2\lceil \beta \rceil+5})]$,
then the Hausdorff dimension of the univoque set is equal to
$(\log(\lceil \beta \rceil-2))(\log\beta)^{-1}$.
\item[(2)] If $\beta\in[(\lceil \beta
\rceil-1+\sqrt{(\lceil \beta \rceil)^2-2\lceil \beta
\rceil+5}),\,\lceil \beta \rceil)$ and a purely Parry number,
Kall\'{o}s can still find the dimensional result.
\end{itemize}
Zou, Lu and Li \cite{ZL} considered the univoque set for
a class of homogeneous self-similar sets with overlaps. Their
motivation was to generalise Glendinning and Sidorov's result
\cite{GS}.
In some cases, they provide an explicit formula for the dimension of the univoque set. What made the work of Zou, Lu and Li different to the work of Glendinning and Sidorov, was that the self-similar sets
they considered were of Lebesgue measure zero.Their approach was similar to Glendinning and Sidorov's, the crucial technique is finding a new characterisation of the univoque set. However, when the similarity ratios change and the attractor becomes an interval, they cannot calculate the dimension of univoque set. Recently, in the setting
of $\beta$-expansions, Kong and Li \cite{DL} generalised Kall\'{o}s'
results, their approach made use of different techniques which were
based on the admissible blocks introduced by Komornik and de Vries
\cite{MK}. They were able to calculate the dimension of the univoque set for $\beta$ within intervals. These intervals cover almost all $\beta$, even some bases for which $\tilde{U}$ is not a subshift of finite type.

In the papers mentioned above, the approaches given always have two
points in common. The first is that their method depends on
finding a symbolic characterisation of the  univoque set via the
greedy algorithm. For general self-similar sets such a
characterisation is not possible. The second point is that in their
setup every first-level interval has at most two adjacent
first-level intervals intersecting it. For general self-similar sets, some first-level intervals
may intersect many first-level intervals simultaneously. As such their methods do not simply translate over and we have to find a new approach.

The goal of this paper is to give a general algorithm for calculating
the Hausdorff dimension of the univoque set when the self-similar
set is an interval. When this algorithm can be implemented it
identifies the univoque set with a subshift of finite type. With this
new symbolic representation, we can use a directed graph to represent
the set $\widetilde{U}$, see for example Chapter 2 \cite{LM}. We then
show that $U$ is a graph-directed self-similar set in
the sense of Mauldin and Williams \cite{MW}. Using the results of
\cite{MW} we can then calculate $\dim_{H}(U)$
explicitly. This algorithm can be implemented in a generic sense
that we will properly formalise later.

The structure of the paper is as follows. In section 2 we describe
the self-similar set via a dynamical system and state Theorem
\ref{MainTheorem} which is our main result. In section 3 we prove
Theorem \ref{MainTheorem} and demonstrate that for most cases, the
hypothesis of Theorem \ref{MainTheorem} is satisfied (Corollary
\ref{Universal coding}). In section 4 we restrict to
$\beta$-expansions and provide an alternative methodology for determining
the subshift of finite type representation of $\widetilde{U}$. In
section 5 we introduce the definition of a graph-directed self-similar
set and illustrate how to calculate the dimension of the univoque
set using this tool. In section 6 we give a worked example. Finally in section 7, we discuss how the approach given can be extended to higher dimension.

After completion of this paper the authors were made aware of the
work of Bundfuss, Kr{\"u}ger and Troubetzkoy \cite{STS}. They were
concerned with iterating maps on a manifold $M$ and the set of $x\in
M$ that were never mapped into some hole. Theorem \ref{MainTheorem}
is essentially a consequence of Proposition 4.1 \cite{STS}. However,
all of our results regarding calculating
$\dim_{H}(U)$ and the identification of the
univoque set with a graph-directed self-similar set are completely
new.

\section{Preliminaries and Main Results}
In this section we describe the elements of our attractor in terms of a dynamical system.
Recall that $K=[a,b]\subseteq\mathbb{R}$ is the attractor of our IFS $\{f_{j}\}_{j=1}^{m}$, i.e.,
\[K=\bigcup_{j=1}^{m}f_j(K).\]
Define $T_{j}(x):=f_{j}^{-1}(x)=(x-a_{j})r_{j}^{-1}$ for each $1\leq j\leq m.$ We denote the concatenation $T_{i_{n}}\circ \ldots \circ T_{i_{1}}(x)$ by $T_{i_1\ldots i_n}(x)$. The following lemma provides an alternative formulation of codings of elements of $K$ in terms of the maps $T_{j}$.
\begin{Lemma}
\label{Dynamical lemma}
Let $x\in K.$ Then $(i_{n})_{n=1}^{\infty}\in\{1,\ldots,m\}^{\mathbb{N}}$ is a coding for $x$ if and only if $T_{i_{1}\ldots i_{n}}(x)\in K$ for all $n\in\mathbb{N}.$
\end{Lemma}
 \begin{proof}
Assume $x\in K$ has a coding $(i_{n})_{n=1}^{\infty}.$ By the continuity of the maps $f_{j}$ the following equation holds for all $n\in\mathbb{N}:$ $$T_{i_{1}\ldots i_{n}}(x)=\lim_{M\to\infty}f_{i_{n+1}}\circ\cdots\circ f_{i_{M}}(0).$$ Obviously the right hand side of the above equation is an element of $K$. As such we have deduced the rightwards implication.

Now let us assume that $(i_{n})_{n=1}^{\infty}$ is such that $T_{i_{1}\ldots i_{n}}(x)\in K$ for all $n\in\mathbb{N}.$ Let $x_{n}=T_{i_{1}\ldots i_{n}}(x).$ We observe the following:
$$|f_{i_{1}}\circ\cdots\circ f_{i_{n}}(0)-x|=|f_{i_{1}}\circ\cdots\circ f_{i_{n}}(0)-f_{i_{1}}\circ\cdots\circ f_{i_{N}}(x_{n})|\leq r^{n}|x_{n}|.$$
 Where $r=\max\limits_{1\leq j\leq m} r_{j}.$ By our assumption $x_{n}\in K,$ in which case $|x_{n}|$ can be bounded above by a constant independent of $x$ and $n.$ It follows that $\lim_{n\to\infty}f_{i_{1}}\circ\cdots\circ f_{i_{n}}(0)=x$ and $(i_{n})_{n=1}^{\infty}$ is a coding for $x.$
\end{proof}
The dynamical interpretation provided by Lemma \ref{Dynamical lemma} will make our proofs and exposition far more succinct. The following proposition is a straightforward consequence of Lemma \ref{Dynamical lemma}.
\setcounter{tmp}{\value{Lemma}}
\setcounter{Proposition}{1}
\begin{Proposition}
\label{non unique prop}
Let $x\in K$.  There exists $(i_{n})_{n=1}^{N}\in\{1,\ldots,m\}^{N}$ and distinct $k,l\in \{1,\ldots,m\}$ satisfying $ T_{i_{1}\cdots i_{N}k}(x)\in K$ and $T_{i_{1}\cdots i_{N}l}(x)\in K$ if and only if  $x\notin U.$
\end{Proposition}
Let $I_{j}=f_{j}(K),$ $I_{j}$ is precisely the set of points that are mapped back into $K$ by $T_{j}.$ The following reformulation of $U$ is a consequence of Proposition \ref{non unique prop}:
\begin{equation}
\label{univoque reformulation}
U=\Big\{x\in K: \nexists 1\leq k<l\leq m \textrm{ and } (i_{n})_{n=1}^{N} \textrm{ such that } T_{i_{1}\cdots i_{N}}(x)\in I_{k}\cap I_{l}\Big\}.
\end{equation}By Lemma \ref{Dynamical lemma} we know that  every $x\in K$ has an infinite sequence of maps which under finite iteration always map $x$ back into $K$. What (\ref{univoque reformulation}) states is that if $x\in U$,  then each of these finite iterations always avoid the intersections of the $I_{j}'s.$

In what follows we always assume that there are $s$ pairs $(i_k,\,j_k)\in\{1,\ldots,m\}^{2}$ such that $H_k:=I_{i_k}\cap I_{j_k}\neq \emptyset$ and $i_{k}\neq j_{k}$. In fact we will always assume that we are in the case where each $H_k:=[a_k,b_k]$ is a nontrivial interval and is contained in the interior of $K$. There is no loss of generality in making this assumption. If for some $[a_k,b_k]$ it is true that $a_k=a$ or $b_k=b$, then the conclusion of Theorem \ref{MainTheorem} is still true under appropriate modified hypothesis. The argument required is the same as that given below except for an additional notational consideration. We may also assume that the elements of $\{H_k\}$ are piecewise disjoint and that they are located from left to right in $K$. In the dynamical literature these regions $H_{k}$ are commonly referred to as switch regions, see for example \cite{KarmaCor}. We give a simple example to illustrate the above.
\setcounter{tmp}{\value{Proposition}}
\setcounter{Example}{2}

\begin{Example}
 Let $[0,1/(\beta-1)]$ be the attractor of $\{f_{0}(x)=\beta^{-1}x,\,f_{1}(x)=\beta^{-1}(x+1)\}$, where $1<\beta<2$. Then we define $T_{0}(x)=\beta x,\,T_{1}(x)=\beta x-1$, see Figure 1.
\begin{figure}[h]\label{figure1}
\centering
\begin{tikzpicture}[scale=5]
\draw(0,0)node[below]{\scriptsize 0}--(.382,0)node[below]{\scriptsize$\frac{1}{\beta}$}--(.618,0)node[below]{\scriptsize$\frac{1}{\beta(\beta-1)}$}--(1,0)node[below]{\scriptsize$\frac{1}{\beta-1}$}--(1,1)--(0,1)node[left]{\scriptsize$\frac{1}{\beta-1}$}--(0,.5)--(0,0);
\draw[dotted](.382,0)--(.382,1)(0.618,0)--(0.618,1);
\draw[thick](0,0)--(0.618,1)(.382,0)--(1,1);
\end{tikzpicture}\caption{The dynamical system for $\{T_{0},\,T_{1}\}$}
\end{figure}
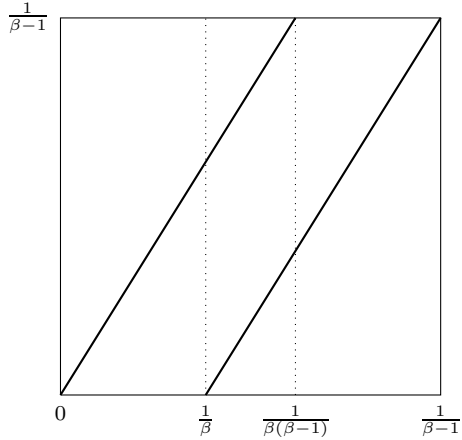

From this figure, we know that $f_{0}([0,1/(\beta-1)])\cap f_{1}([0,1/(\beta-1)])=[1/\beta,\,1/\beta(\beta-1)]$. For any $x\in [1/\beta,\,1/\beta(\beta-1)]$ both $T_0$ and $T_1$ map $x$ into $[0,1/(\beta-1)]$.
\end{Example}
Now we can state our first result. Recall that $\widetilde{U}$ is defined to be the set of symbolic codings of points in $U.$
\setcounter{tmp}{\value{Example}}
\setcounter{Theorem}{3}
\begin{Theorem}\label{MainTheorem}
For each $a_k$ and $b_k$, suppose there exist two finite sequences $(\eta_{1}\ldots \eta_{P})\in\{1,\ldots,\,m\}^{P}$,
$(\omega_{1}\ldots \omega_{Q})\in\{1,\ldots,\,m\}^{Q}$
such that
\begin{eqnarray}\label{Assumption1}
T_{\eta_{1}\ldots \eta_{P}}(a_k)\in \bigcup_{i=1}^{s}(a_i,\,b_i)
\end{eqnarray}
 and
\begin{eqnarray}\label{Assumption2}
T_{\omega_{1}\ldots \omega_{Q}}(b_k)\in \bigcup_{i=1}^{s}(a_i,\,b_i),
\end{eqnarray}
Then $\widetilde{U}$ is a subshift of finite type.
\end{Theorem}
\section{Proof of Theorem \ref{MainTheorem}}
We give a constructive proof of Theorem \ref{MainTheorem}.
\begin{proof}
By our assumptions and the continuity of the $T_{j}$'s, we can find $\delta_{a_k}>0$ and $\delta_{b_k}>0$ such that\[T_{\eta_{1}\ldots \eta_{P}}(a_k-\delta_{a_k}, a_k)\in \bigcup_{i=1}^{s}(a_i,\,b_i)\] and
\[T_{\omega_{1}\ldots \omega_{Q}}(b_k, b_k+\delta_{b_k})\in \bigcup_{i=1}^{s}(a_i,\,b_i).\]
Moreover, we may assume that
$[a_k-\delta_{a_k}, b_k+\delta_{b_k}]\cap [a_j-\delta_{a_j}, b_j+\delta_{b_j}]=\emptyset$ for each $1\leq k<j\leq s.$ Let $\delta=\min\limits_{1\leq k\leq s}\{\delta_{a_k},\,\delta_{b_k}\}$ and $H=\cup_{i=1}^{s}[a_i-\delta, b_i+\delta]$. By the monotonicity of the $T_{j}$'s and Proposition \ref{non unique prop} it is clear that $H$ is in the complement of the univoque set. We partition $K$ via the iterated function system.
For any $L$ we have $$K=\bigcup_{(i_{1},\ldots,i_{L})\in\{1,\ldots,m\}^{L}}f_{i_{1}}\circ\cdots\circ f_{i_{L}}(K).$$ We also assume $L$ is sufficiently large such that $|f_{i_{1}}\circ\cdots\circ f_{i_{L}}(K)|<\delta$ for all $(i_{1},\ldots,i_{L})\in\{1,\ldots,m\}^{L}$. We have a corresponding  partition of the symbolic space $\{1,\ldots,\,m\}^{\mathbb{N}}$ provided by the cylinders of length $L$. For each $(i_{1},\ldots, i_{L})\in\{1,\ldots,m\}^{L}$ let $$C_{i_{1}\ldots i_{L}}=\Big\{(x_{n})\in\{1,\ldots,m\}^{\mathbb{N}}:x_{n}=i_{n} \textrm{ for } 1\leq n\leq L\Big\}.$$ The set $\{C_{i_{1}\ldots {i_{L}}}\}_{(i_{1},\ldots ,i_{L})\in\{1,\ldots,m\}^{L}}$ is a partition of $\{1,\ldots,m\}^{\mathbb{N}},$ and  $f_{i_{1}}\circ\cdots\circ f_{i_{L}}(K)=\pi(C_{i_{1}\ldots {i_{L}}})$.
Let \[\mathbb{F}=\Big\{(i_{1},\ldots,i_{L})\in\{1,\ldots,m\}^{L}:f_{i_{1}}\circ\cdots\circ f_{i_{L}}(K)\cap \bigcup_{k=1}^{s} H_k\neq\varnothing\Big\}\]
and
\[\mathbb{F^{'}}=\bigcup_{(i_{1},\ldots,i_{L})\in \mathbb{F}}\pi(C_{i_{1}\ldots {i_{L}}}).\]
 By our assumptions on the size of our cylinders the following inclusions hold \[\bigcup_{k=1}^{s} H_k\subset \mathbb{F^{'}}\subset H.\]
Using these inclusions it is a straightforward observation that $x\notin U$ if and only if there exists $(\theta_1,\ldots,\theta_{n_1})\in\{1,\ldots,m\}^{n_1}$ such that  $T_{\theta_1\ldots\theta_{n_1}}(x)\in \mathbb{F^{'}}$. Showing there exists $(\theta_1,\ldots,\theta_{n_1})\in\{1,\ldots,m\}^{n_1}$ such that  $T_{\theta_1\ldots\theta_{n_1}}(x)\in \mathbb{F^{'}}$ if and only if $x$ has a coding containing a block from $\mathbb{F}$ is straightforward. Therefore if we take $\mathbb{F}$ to be the set of forbidden words defining a subshift of finite type we see that $\widetilde{U}$ is a subshift of finite type.
\end{proof}
The conditions in Theorem \ref{MainTheorem} are met for a large class of self-similar sets, provided that the attractor is an interval. We recall the definition of a universal coding. A  coding $(d_n)_{n=1}^{\infty}\in\{1,\ldots,m\}^{\mathbb{N}}$ is called a universal coding for $x$ if  given any finite block $(\delta_1,\ldots, \delta_k)\in \{1,\ldots,m\}^{k}$, there exists $j$ such that $d_{j+i}=\delta_i$ for $1\leq i\leq k$. Theorem 1.4 from \cite{BS} implies that Lebesgue almost every $x\in K$ has a universal coding. This result implies the following corollary.
\begin{Corollary}\label{Universal coding}
For Lebesgue almost every $x\in K$, there exists a sequence $(i_{n})_{n=1}^{N}$ and $H_{k}$ such that $T_{i_{1}\ldots i_{N}}(x)$ is in the interior of $H_{k}$.
\end{Corollary}
By this corollary, it follows that the hypothesis of Theorem \ref{MainTheorem} are failed only when an endpoint of a $H_{k}$'s is contained in a set of Lebesgue measure zero. As such the conditions of Theorem \ref{MainTheorem} hold in a generic sense. As we will see in section $4,$ a stronger statement holds when we restrict to $\beta$-expansion.
\setcounter{tmp}{\value{Corollary}}
\setcounter{Remark}{1}
\begin{Remark}
In \cite{DKK},\cite{K1},\cite{K2} and \cite{DL}, they all consider homogeneous IFS's. We however allow the similarity ratios to be different. Another advantage of our method is that for different IFS's, we can find the forbidden blocks quickly and uniformly.
\end{Remark}
\begin{Remark}\label{General self-similar sets}
 The method used in Theorem  \ref{MainTheorem} cannot easily be implemented when $K$ is not an interval. The key difficulty is that when we construct the neighborhoods of $a_k$ and $b_k$, the images of these neighborhoods may not be mapped into $\cup_{k=1}^{s}H_k$ by the same maps that worked for $a_{k}$ and $b_{k}$. 
\end{Remark}
\begin{Remark}\label{Highdimension}
In higher dimensions we can prove an analogous result. The proof requires a minor modification, the main ideas are outlined in the final section. For self-affine sets which are simple sets, for instance,  rectangles, cubes(see the definition of self-affine sets in \cite{FG}), our theorem still holds. However, in this case we do not know whether an analogue of Corollary \ref{Universal coding} is true.
\end{Remark}

 Using a similar idea to the proof of Theorem \ref{MainTheorem}, we can prove following theorem:
\setcounter{tmp}{\value{Remark}}
\setcounter{Theorem}{4}
\begin{Theorem}\label{Closedset}
For any arbitrary interval $K$, $U$ is closed if and only if $\widetilde{U}$ is a subshift of finite type.
\end{Theorem}
This theorem generalises Komornik and de Vries' statement, see the corresponding  equivalent statements in \cite[Theorem 1.8]{MK}. Moreover, in higher dimensions similar result still holds.

\section{$\beta$-expansions case}
In this section we restrict to $\beta$-expansions and give an alternative method for determining the subshift of finite type representation of $\widetilde{U}.$
Firstly, we recall the relevant IFS for studying $\beta$-expansions. Given $\beta>1$ define the IFS:
\[g_j(x)=\dfrac{x+j}{\beta},\,j\in\{0,\ldots,\lceil \beta \rceil-1\}.\]
The self-similar set for this IFS is the interval
$\mathcal{A}_{\beta}=[0,(\lceil \beta \rceil-1)(\beta-1)^{-1}]$.

We now define greedy and lazy expansions.
\begin{Definition}
The greedy map $G:\mathcal{A}_{\beta}\to \mathcal{A}_{\beta},$ is defined by
\begin{equation*}
G(x)=\left\lbrace\begin{array}{cc}
                 \beta x\mod 1& x\in[0,1)\\

                 \beta x-[\beta]& x\in\Big[1, \frac{\lceil \beta \rceil-1}{\beta-1}\Big]
                \end{array}\right.
\end{equation*}
\end{Definition}
For any $n\geq 1$ and $x\in\mathcal{A}_{\beta}$, we define $a_n(x)=[\beta G^{n-1}(x)]$, where $[y]$ denotes the integer part of $y\in\mathbb{R}$.
We then  have
\begin{eqnarray*}
x&=&\dfrac{a_{1}(x)}{\beta}+\dfrac{G(x)}{\beta}\\
  &=&\dfrac{a_{1}(x)}{\beta}+\dfrac{a_2(x)}{\beta^2}+\dfrac{G^2(x)}{\beta^2}\\
&&\vdots\\
&=&\sum_{n=1}^{\infty}\dfrac{a_n(x)}{\beta^n}
\end{eqnarray*}
The sequence  $(a_n)_{n=1}^{\infty}\in\{0,\ldots,\lceil \beta \rceil-1\}^{\mathbb{N}}$ generated by $G$ is called the greedy expansion or greedy coding. The orbit $\{G^{n}(x)\}_{n=1}^{\infty}$ is called the greedy orbit of $x$.

Similarly, we define the lazy map and the corresponding lazy expansion as follows.
\begin{Definition}
The lazy map $L:\mathcal{A}_{\beta}\to \mathcal{A}_{\beta},$ is defined by

\begin{equation*}
L(x)=\left\lbrace\begin{array}{cc}
                 \beta x& x\in\Big[0,\frac{(\lceil \beta \rceil-1)}{(\beta(\beta-1))}\Big]\\

                 \beta x-b_{j}& x\in \Big( \frac{\lceil \beta \rceil-1}{\beta(\beta-1)}+\frac{b_{j}-1}{\beta}, \frac{\lceil \beta \rceil-1}{\beta(\beta-1)}+\frac{b_{j}}{\beta}\Big]\textrm{ for }b_{j}\geq 1
                \end{array}\right.
\end{equation*}

By Lemma \ref{Dynamical lemma}, for each $x\in\mathcal{A}_{\beta}$ we can generate a $\beta$-expansion for $x$ by iterating $L$. The $\beta$-expansion generated by $L$ is called the lazy expansion of $x$. The orbit $\{L^{n}(x)\}_{n=1}^{\infty}$ is called the lazy orbit of $x$.
\end{Definition}
 Given $i\in\{0,\ldots ,\lceil \beta \rceil -1\}$ it is a simple calculation to show that $g_{i}(\mathcal{A}_{\beta})\cap g_{j}(\mathcal{A}_{\beta})\neq\varnothing$ if and only if $j=i-1,i,i+1$. In which case the nontrivial switch regions are of the form:
\[
S_{l}=\Big[\frac{l}{\beta},\frac{\lceil \beta \rceil-1}{\beta(\beta-1)}+\frac{l-1}{\beta}\Big],
\]
for some $1\leq l\leq \lceil \beta \rceil-1$. We remark that the greedy and lazy maps only differ on the intervals $S_{l}.$ Clearly an $x\in\mathcal{A}_{\beta}$ is a univoque point if and only if it is never mapped into an interval $S_{l}.$ This implies the following important technical result.
\setcounter{tmp}{\value{Definition}}
\setcounter{Proposition}{2}
\begin{Proposition}
Given $x\in K,$ we have that $x\in U$ if and only if its greedy and lazy expansions coincide.
\end{Proposition}
This simple observation will be a powerful tool, it allows us to give a lexicographic characterisation of $\widetilde{U}$ which will help us determine our subshift of finite type representation.

Each element of $U\setminus\{0,(\lceil \beta \rceil-1)(\beta-1)^{-1}\}$ is eventually mapped into $[(\lceil \beta \rceil-1-\beta)(\beta-1)^{-1},1]$ by $G$ and $L$ (as by definition the orbits of $G$ and $L$ coincide for univoque points). Moreover, once inside this interval they are not mapped out, see  \cite[page 2]{GS}. Therefore, due to the countable stability of Hausdorff dimension(\cite[page 32]{FG}), to determine the Hausdorff dimension of $U$, we only need to find the Hausdorff dimension of $U\cap[(\lceil \beta \rceil-1-\beta)(\beta-1)^{-1},1]$. We denote $U\cap[(\lceil \beta \rceil-1-\beta)(\beta-1)^{-1},1]$ and $\pi^{-1}(U\cap [(\lceil \beta \rceil-1-\beta)(\beta-1)^{-1},1])$ by $U_{\beta}$ and $\widetilde{U}_{\beta}$ respectively.

Let $(\alpha_n)_{n=1}^{\infty}$ be the greedy expansion of $1$ and $(\varepsilon_n)_{n=1}^{\infty}=(\overline{\alpha_n})_{n=1}^{\infty}=(\lceil \beta \rceil-1-\alpha_n)_{n=1}^{\infty}$.
We are interested in giving conditions when $\widetilde{U}_{\beta}$ is a subshift of finite type. In this paper we consider only the collection of $\beta$ such that  the greedy expansion of 1 is infinite. If the greedy expansion of 1 is finite, then $\widetilde{U}_{\beta}$ may be not a subshift of finite type, the good examples are Tribonacci numbers, see Theorems 1.2 and 1.5 from \cite{MK}. 
 Let $\sigma$ denote the usual shift map. We now introduce the lexicographic ordering on infinite sequences, given $(a_{n})_{n=1}^{\infty},(b_{n})_{n=1}^{\infty}\in\{0,\ldots, \lceil \beta \rceil -1 \}^{\mathbb{N}}$ we say that  $(a_n)_{n=1}^{\infty}<(b_n)_{n=1}^{\infty}$ if there exists $M\in \mathbb{N}$ such that $(a_1,\ldots, a_M)=(b_1,\ldots, b_M)$ and $a_{M+1}<b_{M+1}$. There also exists a lexicographic ordering on finite sequences, this is defined in the obvious way.
\setcounter{tmp}{\value{Proposition}}
\setcounter{Theorem}{3}
\begin{Theorem}\label{MainTheorem1}
If there exists $M\in\mathbb{N}$ such that $(\varepsilon_{M+n})_{n=1}^{\infty}>(\alpha_n)_{n=1}^{\infty}$ then $\widetilde{U}_{\beta}$ is a subshift of finite type. More specifically, there exists $p>M$ such that
\begin{align*}
\widetilde{U}_{\beta}=\Big\{(d_n)_{n=1}^{\infty}:(\varepsilon_1,\ldots,\varepsilon_p,(\lceil \beta \rceil-1)^{\infty})<\sigma^{k}((d_n)_{n=1}^{\infty})<(\alpha_1,\ldots,\alpha_p,(0)^{\infty}) \mbox{for any } k\geq 0\Big\}.
\end{align*}

\end{Theorem} The hypothesis of Theorem \ref{MainTheorem1} is in fact equivalent to that of Theorem \ref{MainTheorem}. We omit the details of this equivalence as it hinders our exposition. The spirit of this proof is similar to the proof of Theorem \ref{MainTheorem}. Heuristically speaking, we are giving an equivalent argument but expressed in the language of sequences. When expressed in this language the proof becomes more concise and provides a more efficient method for determining the set of forbidden words.

The following criterion  of the unique codings is pivotal. In fact, in \cite{DKK}, \cite{K1}, \cite{K2} and \cite{DL}, their approaches strongly depend on this criterion.
\begin{Theorem}\label{Uniquecodings}
 Let $(a_n)_{n=1}^{\infty}$ be a coding of $x\in [(\lceil \beta \rceil-1-\beta)(\beta-1)^{-1},1]$. Then $(a_n)_{n=1}^{\infty}\in \widetilde{U}_{\beta}$ if and only if \[(\varepsilon_n)_{n=1}^{\infty}<\sigma^{k}((a_n)_{n=1}^{\infty})<(\alpha_n)_{n=1}^{\infty}\] for any $k\geq 0$.
\end{Theorem}
This theorem is a corollary of Theorem 1.1 \cite{MK}.
\begin{proof}[Proof of  Theorem \ref{MainTheorem1}]
From Theorem \ref{Uniquecodings} we know that  \[\widetilde{U}_{\beta}=\{(a_n)_{n=1}^{\infty}:(\varepsilon_n)_{n=1}^{\infty}<\sigma^{k}((a_n)_{n=1}^{\infty})<(\alpha_n)_{n=1}^{\infty}\,\mbox{for any}\,k\geq 0\}.\]
Let $M$ be as in the statement of Theorem \ref{MainTheorem1}, there exists $p>M$ such that  \[(\varepsilon_{M+1},\ldots,\varepsilon_{p})>(\alpha_1,\ldots,\alpha_{p-M}).\]
Recall $(\varepsilon_n)=(\overline{\alpha_n})$, thus we equivalently have \[(\varepsilon_{1},\ldots,\varepsilon_{p-M})>(\alpha_{M+1},\ldots,\alpha_{p}).\]
We shall prove that $\widetilde{U}_{\beta}=U^{'}_{\beta}$ where \[U^{'}_{\beta}:=\left\{(a_n)_{n=1}^{\infty}:(\varepsilon_1,\ldots,\varepsilon_p,(\lceil \beta \rceil-1)^{\infty})<\sigma^{k}((a_n)_{n=1}^{\infty})<(\alpha_1,\ldots,\alpha_p,(0)^{\infty})\,\mbox{for any }\, k\geq 0\right\}.\]
By Theorem \ref{Uniquecodings} we have $U^{'}_{\beta}\subseteq \widetilde{U}_{\beta}$, therefore it suffices to prove the opposite inclusion.

Let $(a_n)_{n=1}^{\infty}\in \widetilde{U}_{\beta}$ and assume that $(a_n)_{n=1}^{\infty}\notin U^{'}_{\beta}$. Therefore, we have $\sigma^{k_0}((a_n)_{n=1}^{\infty})\geq(\alpha_1,\ldots,\alpha_p,(0)^{\infty})$ or $(\varepsilon_1,\ldots,\varepsilon_p,(\lceil \beta \rceil-1)^{\infty})\geq\sigma^{k_0}((a_n)_{n=1}^{\infty})$ for some $k_0\geq 0$. But this is not possible. For instance, if $(\varepsilon_1,\ldots,\varepsilon_p,(\lceil \beta \rceil-1)^{\infty})\geq\sigma^{k_0}((a_n)_{n=1}^{\infty})$ then  $(a_{k_0+1},\ldots, a_{k_0+p})=(\varepsilon_1,\ldots,\varepsilon_p)$ since $(a_n)_{n=1}^{\infty}\in \widetilde{U}_{\beta}.$ Hence,
\[(a_{k_0+M+1},\ldots, a_{k_0+p})=(\varepsilon_{M+1},\ldots,\varepsilon_p)>(\alpha_1,\ldots,\alpha_{p-M})\]
but this contradicts the fact that $(a_n)_{n=1}^{\infty}\in \widetilde{U}_{\beta}$. The other case is proved similarly. As such we may conclude that $\widetilde{U}_{\beta}\subseteq U^{'}_{\beta}.$
\end{proof}
\setcounter{tmp}{\value{Theorem}}
\setcounter{Remark}{5}
\begin{Remark}
Theorem \ref{MainTheorem1} implies that when the greedy orbit of $1$ falls into the interior of the switch region, then $\widetilde{U}_{\beta}$ is a subshift of finite type. This theorem is a little weaker than Komornik and de Vries' statement, see \cite[Theorem 1.8]{MK}. However, we can find the forbidden blocks more quickly. It is not necessary to use Theorem \ref{Uniquecodings} to find the subshift of finite type, while Komornik and de Vries' method depends on it. We have proved in Theorem \ref{MainTheorem} that for self-similar sets a similar idea  still works. Moreover, we mentioned in Theorem  \ref{Closedset} that  for any arbitrary interval $K$, $U$ is closed if and only if $\widetilde{U}$ is a subshift of finite type.  As such Theorem \ref{MainTheorem} can be interpreted as a generalisation of Komornik and de Vries' result to the setting of self-similar sets.
\end{Remark}
\begin{Remark}
In \cite{K2}, Kall\'{o}s used similar ideas to prove a similar theorem. However, the argument in the proof of Theorem \ref{MainTheorem1} may not be applied in other complicated settings as  generally we cannot find a criteria for unique codings in terms of a symbolic representation.
\end{Remark}
In the setting of $\beta$-expansions, let \[A=\{\beta\in(1,\infty):\mbox{ 1 has a unique expansion in base}\, \beta\}.\] Schmeling \cite{Schmeling} (also see Daroczy and Katai \cite{DKK}) proved the Lebesgue measure of $A$ is zero. In fact Schmeling proved a much stronger result. This statement implies the following corollary.
\setcounter{tmp}{\value{Remark}}
\setcounter{Corollary}{7}
\begin{Corollary}
\label{beta corollary}
For almost every $\beta\in(1,\infty)$ the hypothesis of Theorem \ref{MainTheorem1} are satisfied.
\end{Corollary}This should be compared with Corollary \ref{Universal coding}. We see that Corollary \ref{beta corollary} allows us to conclude a stronger result in the setting of $\beta$-expansions.

\section{Hausdorff dimension of univoque set}
\subsection{Graph-directed self-similar sets}
Before  demonstrating how to calculate the dimension of a univoque set, we introduce the notion of a graph-directed self-similar set. The terminology we use is taken from \cite{MW}.

A graph-directed construction in $\mathbb{R}$ consists of the following.
\begin{enumerate}
\item A finite union of bounded closed intervals $\cup_{u=1}^{n}J_u$ such that the $J_u$ are piecewise disjoint.
\item A directed graph $G=(V,E)$ with vertex set $V=\{1,\ldots,n\}$ and edge set $E.$ Moreover, we assume that for any $u\in V$ there is some $v\in V$ such that $(u,v)\in E.$
\item For each edge $(u,v)\in E$ there exists a similitude $f_{u,v}(x)=r_{uv}x+a_{uv}$, where $r_{uv}\in (0,1)$ and $a_{uv}\in \mathbb{R}$.  Moreover, for each $u \in V$ the set $\{f_{u,v}(J_v):(u,v)\in E\}$ satisfies the strong separation condition, i.e.,
\[\bigcup_{(u,v)\in E}f_{u,v}(J_v)\subseteq J_u,\] and the elements of $\{f_{u,v}(J_v):(u,v)\in E\}$ are piecewise disjoint.

\end{enumerate}
As is the case for self-similar sets, we have the following result.
\begin{Theorem}\label{Existence}
 For each graph-directed construction, there exists a unique vector of non-empty compact sets $(C_1,\ldots,C_n)$ such that, for each $u\in V$, $C_{u}=\bigcup_{(u,v)\in E}f_{u,v}(C_v)$.
\end{Theorem}
 We let $K^{*}:=\cup_{u=1}^{n}C_u$ and call it the graph-directed self-similar set of this construction.
To each graph-directed construction we can associate a weighted incidence matrix $A$. This matrix is defined by $A=(r_{u,v})_{(u,v)\in V\times V}$, for simplicity, we assume that $r_{u,v}=0$ if $(u,v)\notin E$. For each $t\geq 0$ we define another adjacency matrix $A^t=(a_{t,u,v})_{(u,v)\in V\times V},$ where $a_{t,u,v}=r_{u,v}^t$. Let $\Phi(t)$ denote the largest nonnegative eigenvalue of $A^t$.
A graph is strongly connected if for any two vertices $u,v\in V$, there exists a directed path from $u$ to $v$. A strongly connected component of $G$ is a subgraph $C$ of $G$ such that $C$ is strongly connected, let $SC(G)$ be the set of all the  strongly connected components of $G$.
Now we state the main result of \cite{MW}.
\begin{Theorem}\label{SC}
 For every graph-directed construction such that $G$ is strongly connected, the Hausdorff dimension of $K^{*}$ is $t_0$, where $t_0$ is uniquely defined by $\Phi(t_0)=1$.
\end{Theorem}
 If the  graph-directed construction $G$ is not strongly connected, we still have a similar result. As is well known, a directed graph $G$ must have a strongly connected component, see \cite[section 4.4.]{LM}. In which case the following theorem makes sense.
\begin{Theorem}\label{NSC}
If the $G$ in our graph-directed construction is not strongly connected, let $t_1=\max\{t_C: \Phi(t_C)=1,\, C\in SC(G)\}$, where $\Phi(t_C)$ is the largest eigenvalue of the adjacency matrix of the strongly connected subgraph $C$. Then $\dim_{H}(K^{*})=t_1$.
\end{Theorem}
\begin{proof}
 We can decompose $G$ into several subgraphs which are each strongly connected, then this theorem holds due to Theorem \ref{SC} and the countable stability of Hausdorff dimension.
\end{proof}
\subsection{Calculating the dimension of univoque set}
Now we show how to construct a graph-directed self-similar set using the subshift of finite type representation of $\widetilde{U}$ obtained in Theorem \ref{MainTheorem}. As we will see, in this case, the graph-directed self-similar set $K^{*}$ mentioned above will in fact equal $U$.

Recall the projection map $\pi:\{1,\ldots,m\}^{\mathbb{N}}\to K$ is defined by
\[\pi((i_n)_{n=1}^{\infty})=\lim\limits_{n\to \infty}f_{i_1}\circ\cdots\circ f_{i_n}(0).\]
 We use the same notation as in the proof of Theorem \ref{MainTheorem}. Let $\mathbb{F}$ be the set of finite forbidden blocks and $W=\{1,\ldots,m\}^{L}\setminus \mathbb{F}$. The set of vertices in our directed graph will be:
\begin{align*}
V=\Big\{(a_1,\ldots, a_{L-1})\in \{1,\ldots,m\}^{L-1}:&\,\mbox{there exists}\,a_L\in\{1,\ldots,m\}\, \mbox{such that}\\
& (a_1,\ldots, a_{L-1},a_{L})\in W\Big\}
\end{align*}
We now define our edges. For any two vertices $u,\,v\in V$,
$u=(u_1,\ldots, u_{L-1}),\, v=(v_1,\ldots v_{L-1})$, we draw an edge from $u$ to $v$ and label this edge $(u,v)$, if $(u_{2},\ldots, u_{L-1})=(v_1,\ldots
v_{L-2})$ and $(u_1,\ldots, u_{L-1},v_{L-1})\in W$.  Here we should note that the vertices $u,v$ which are from $V$ are blocks, while in the definition of a graph-directed construction $u$ and $v$ refer to integers.

Now we have defined our edges and hence we have constructed a directed graph $G=(V,\, E).$ If there exists a vertex $u\in V$ for which there is no $v\in V$ satisfying $(u,v)\in E,$ then we remove $u$ from our vertex set. Removing this $u$ does not change any of the latter results, so without loss of generality we may assume that for every $u\in V$ there exists $v\in V$ for which $(u,v)$ is an allowable edge. In which case we satisfy $2.$ in the above definition of a graph-directed construction.

  Before showing that we satisfy $1.$ and $3.$ in the definition of a graph-directed construction we recall an important result from \cite{LM}. We define an infinite path in our graph $G$ to be a sequence $((u^{n},v^{n}))_{n=1}^{\infty}\in E^{\mathbb{N}}$ such that  $v^{n}=u^{n+1} $ for all $n\in \mathbb{N},$ where $u^{n}=u_{1}^{n}u_{2}^{n}\cdots u_{L-1}^{n}$
Define $X_{G}$ to be
\begin{align*}
X_{G}:=\Big\{(y_{n})_{n=1}^{\infty}\in\{1,\ldots,m\}^{\mathbb{N}}:& \textrm{ there exists an infinite path } ((u^{n},v^{n}))_{n=1}^{\infty}\in E^{\mathbb{N}} \textrm{ such that } \\
&y_{n}=u^{n}_{1} \textrm { for all }n\in\mathbb{N}\Big\}.
\end{align*}

Theorem 2.3.2 of \cite{LM} states the following.

\begin{Theorem}\label{LM theorem}
Let $G$ be the directed graph as constructed above. Then $\widetilde{U}=X_{G}.$
\end{Theorem}

We define
$$K_{u}:=\Big\{x=\pi((d_n)_{n=1}^{\infty}):d_i=u_i \textrm{ for }1\leq i\leq L-1 \textrm{ and } (d_{n})_{n=1}^{\infty}\in \widetilde{U} \Big\}$$
and $$J_{u}:=\textrm{conv}(K_{u}).$$ Here $u=(u_{1},\ldots, u_{L-1})\in V$ and $\textrm{conv}(\cdot)$ denotes the convex hull.
\setcounter{tmp}{\value{Theorem}}
\setcounter{Lemma}{4}
\begin{Lemma}
\label{part 1 lemma}
Let $u,v\in V$ and $u\neq v.$ Then $J_{u}\cap J_{v}=\emptyset.$
\end{Lemma}
\begin{proof}
$J_{u}$ and $J_{v}$ are the convex hulls of $K_u$ and $K_v$ respectively, as such they are both intervals. We assume that $J_{u}=[c,d]$ and $J_{v}=[e,f]$. As $K_{u}$ is compact, the endpoints of $J_{u}$ are elements of $K_u$. Similarly, $e,f\in K_v$. Now we prove that $[c,d] \cap [e,f]=\varnothing$.

If $[c,d]$ and $[e,f]$ intersect in a point then this point must be an endpoint.  Without loss of generality assume $d=e$, then $d\in K_u \cap K_v$. However, $K_{u}\subset U$ and we have a contradiction as $u\neq v$.

Now let us assume $J_{u}$ and $J_{v}$ intersect in an interval. Without loss of generality, we assume that $c<e<d$. Since $e$ is a univoque point in $K_v$, we know by Proposition \ref{non unique prop} that there exists a unique sequence of $T_{j}$'s of length $L-1$ that map $e$ into $K$. As $e\in K_{v}$ this sequence of transformations must be $T_{v_{1}\cdots v_{L-1}}.$ By our assumption $c<e<d,$ therefore by the monotonicity of the maps $T_{j}$ we have that $T_{u_{1}\cdots u_{L-1}}(c)<T_{u_{1}\cdots u_{L-1}}(e)< T_{u_{1}\cdots u_{L-1}}(d)$. Both $T_{u_{1}\cdots u_{L-1}}(c),T_{u_{1}\cdots u_{L-1}}(d)\in K,$ but as $K$ is an interval this implies $T_{u_{1}\cdots u_{L-1}}(e)\in K,$ a contradiction.

\end{proof}
By Lemma \ref{part 1 lemma} we can take $\{J_{u}\}_{u\in V}$ to be the set of bounded closed intervals required in $1.$ of the definition of a graph-directed construction.

It remains to show that we satisfy $3.$ of the definition of a graph-directed construction. First of all we define our similitudes, given an edge $(u,v)\in E$ we define $f_{uv}(x)=r_{u_1}x+a_{u_1}.$ The following lemma proves that we satisfy $3.$  of the  graph-directed construction.
\begin{Lemma}\label{part 2 lemma}
 Fix $u\in V.$ Then $\bigcup\limits_{(u,v)\in E}f_{uv}(J_v)\subseteq  J_u$
 and $f_{uv}(J_v) \cap f_{uv^{'}}(J_{v^{'}})=\varnothing$,
for all distinct pairs of edges.
\end{Lemma}
\begin{proof}
 For the first statement, it is sufficient to prove $\bigcup\limits_{(u,v)\in E}f_{uv}(K_v)\subseteq K_u$.
 Suppose $(u,v)\in E$ and $x=f_{uv}(y)$ where $y\in K_{v}$. Let $(y_{n})_{n=1}^{\infty}\in \widetilde{U}$ be the unique coding of $y.$ By Theorem \ref{LM theorem} we know that $(y_{n})_{n=1}^{\infty}\in X_{G}.$ Let $(x_{n})_{n=1}^{\infty}$ be such that $x_{1}=u_{1}$ and $x_{i}=y_{i-1}$ for $i\geq 2,$ then $(x_{n})_{n=1}^{\infty}$ is a coding of $x.$ Since $(u,v)\in E$ we have that $(u_{2},\ldots, u_{L-1})=(v_{1},\ldots, v_{L-2}).$ Moreover, as $(u,v)\in E$ and $(y_{n})_{n=1}^{\infty}\in X_{G}$ then $(x_{n})_{n=1}^{\infty}\in X_{G}$. Using Theorem
\ref{LM theorem} again we know that $(x_{n})_{n=1}^{\infty}\in \widetilde{U},$ which combined with the observation $(x_{1},\ldots, x_{L-1})=(u_{1},\ldots, u_{L-1})$ implies $x\in K_{u}.$   

 The second statement is an immediate consequence of Lemma \ref{part 1 lemma} and the fact that our similitudes are bijections from $\mathbb{R}$ to $\mathbb{R}$ that do not depend on $v$.

\end{proof}
We have satisfied all of the criteria for a graph-directed construction and may therefore conclude that Theorem \ref{Existence} holds. We now show that for our graph construction $K^{*}=U.$ We begin by showing that the $K_{u}$'s are precisely the $C_{u}$'s in Theorem \ref{Existence}.
\begin{Lemma}\label{Graph}
For each $u\in V$ we have $K_u=\bigcup\limits_{(u,v)\in E}f_{uv}(K_v)$.
\end{Lemma}
\begin{proof}
Let $x\in K_{u}$ and $(x_{n})_{n=1}^{\infty}$ be the unique coding for $x$. Then
 $x_n=u_n$ for $1\leq n \leq L-1$. Let \[v=(v_1,\ldots,
v_{L-1})=(x_2,\ldots, x_{L})=(u_2,\ldots, u_{L-1},x_{L}).\] By Theorem \ref{LM theorem} we have $(x_{n})_{n=1}^{\infty}\in X_{G}.$ Therefore $v\in V$ and  $(u,v)\in E$.
Let $y\in K$ have coding $(x_{n+1})_{n=1}^{\infty},$ $(x_{n+1})_{n=1}^{\infty}\in X_{G}$ and by Theorem \ref{LM theorem} we know that $(x_{n+1})_{n=1}^{\infty}\in \widetilde{U}.$ As $(x_{2},\ldots,x_{L})=(v_{1},\ldots, v_{L-1})$ we can deduce that $y\in K_{v}$. As $f_{uv}(y)=x$ we have shown that $K_u\subseteq\bigcup\limits_{(u,v)\in E}f_{uv}(K_v)$. The inverse inclusion is proved in Lemma \ref{part 2 lemma}.
\end{proof}

By the uniqueness part of Theorem \ref{Existence} we may conclude from Lemma \ref{Graph} that the set $\cup_{u=1}^{n}C_{u}$ in the statement equals $\cup_{u\in V}K_{u}.$ The fact that $U=\cup_{u\in V} K_{u}$ is immediate from the definition of $K_{u}.$ As such $U=K^{*}$ and is the graph directed self-similar set for our construction. Therefore, Theorem \ref{SC} and Theorem \ref{NSC} apply and we use them to calculate the Hausdorff dimension of $U$. We include an explicit calculation in section 6.

Now we give a final remark to finish this section.
In \cite{DL} Kong and Li proved the following interesting result.
\setcounter{tmp}{\value{Lemma}}
\setcounter{Theorem}{7}
\begin{Theorem}\label{KongLi}
There exists intervals for which the function mapping $\beta$ to the Hausdorff dimension of the univoque set is strictly decreasing.
\end{Theorem}
This result is somewhat counterintuitive. As $\beta$ gets the larger, the corresponding switch regions shrink. As such, one might expect that the set of points whose orbits avoid the switch regions, i.e. the univoque set, would be larger. However, Theorem \ref{KongLi} shows that in terms of Hausdorff dimension this is not always the case.

A similar idea to the proof of Theorem \ref{MainTheorem} allows us to recover Theorem \ref{KongLi} quickly. We only give an outline of this argument. A straightforward manipulation of the formulas given in Theorem \ref{SC} and Theorem \ref{NSC}, yields that $\dim_{H}(U_{\beta})=\dfrac{\log \lambda}{\log\beta}$, where $\lambda$ is the largest eigenvalue of the transition matrix defining our subshift of finite type. Using similar ideas to those given in the proof of Theorem \ref{MainTheorem}, we can show that if $\beta$ satisfies the hypothesis of Theorem \ref{MainTheorem}, then the hypothesis is also satisfied for $\beta'$ sufficiently close to $\beta.$ Moreover, a more delicate argument implies that  for $\beta'$ sufficiently close to $\beta,$ the set of forbidden words for $\beta'$ equals the set of forbidden words for $\beta$. In other words, the subshift of finite type defining the univoque set for $\beta'$ equals the subshift of finite type defining the univoque set for $\beta$. Observing that $\dim_{H}(U_{\beta})$ is decreasing on some sufficiently small interval containing $\beta$ now follows from the formula stated above.

\section{Example}
In this section, we give an example to show how to calculate the dimension of a univoque set.
\begin{Example}
Let $[0,\frac{1}{\beta-1}]$ be the self-similar set  with IFS:
$\{f_{0}(x),\,f_{1}(x)\}$ where
\[
f_{0}(x)=\dfrac{x}{\beta},\,f_{1}(x)=\dfrac{x+1}{\beta}
\]
Let  $\beta^{*}$ be the unique $\beta\in(1,2)$ satisfing the equation $(111(00001)^{\infty})_{\beta}=1$. In this case $\beta^{*}\approx 1.84$. We now calculate $\dim_{H}(U_{\beta^{*}})$.
\end{Example}
The greedy expansion of $1$ in this base is $(\alpha_n)_{n=1}^{\infty}=(111(00001)^{\infty})$, we observe that  $(\varepsilon_4,\varepsilon_5,\varepsilon_6,\varepsilon_7)>(\alpha_1,\alpha_2,\alpha_3,\alpha_4)$. In which case, by Theorem \ref{MainTheorem1} we have that $\widetilde{U}$ is given by a subshift of finite type. Moreover, in the statement of Theorem \ref{MainTheorem1} we can take $p=7$. We now construct the relevant directed graph. In this case our set $W$ is:
\[W=\{(a_{1},\ldots
,a_{7}):(0001111)<(a_{1},\ldots, a_{7})<(1110000)\},\] moreover the set of vertices equals
\[V=\{(a_{1},\ldots
,a_{6}):(000111)<(a_{1},\ldots, a_{6})<(111000)\}.\]
We now construct the edge set in accordance with the construction given in section 5.2. In total there are 26 vertices:
\[v_1=(001001)\,
v_{2}=(001010)\,
v_{3}=(001011)\,
v_{4}=(001100)\,
v_{5}=(001101)\,\]
\[v_{6}=(010010)\,
v_{7}=(010011)\,
v_{8}=(010100)\,
v_{9}=(010101)\,
v_{10}=(010110)\,\]
\[v_{11}=(011001)\,
v_{12}=(011010)\,
v_{13}=(011011)\,
v_{14}=(100100)\,
v_{15}=(100101)\,\]
\[v_{16}=(100110)\,
v_{17}=(101001)\,
v_{18}=(101010)\,
v_{19}=(101011)\,
v_{20}=(101100)\,\]
\[v_{21}=(101101)\,
v_{22}=(110010)\,
v_{23}=(110011)\,
v_{24}=(110100)\,
v_{25}=(110101)\,\]
\[v_{26}=(110110).\]
We now follow Mauldin and William approach and construct a $26\times 26$ matrix $(A_{i,j}),$ where $A_{i,j}=\dfrac{1}{(\beta^{*})^{t}}$ if there is an edge from vertex $v_i$ to $v_j$, otherwise, $A_{i,j}=0$.
A computer calculation then yields $\dim_{H}(U_{\beta^{*}})\approx0.79$.

\section{Final remark}
We mentioned in Remark \ref{Highdimension} that the main idea of Theorem \ref{MainTheorem} is still effective in higher dimensions. To conclude we give a brief outline of the argument required. First of all assume that our attractor $K\subset \mathbb{R}^{d}$ is some sufficiently nice set, i.e. a rectangle, cube, polyhedra. In which case the switch regions are also nice sets. We assume that every point on the boundary of the switch regions is mapped into the interior of a switch region. An analogue of Corollary \ref{Universal coding} holds in higher dimensions, as such we expect this assumption to hold generically.  As a consequence of this construction we can enlarge the switch region and not change the univoque set. A similar argument to that given in the proof of Theorem \ref{MainTheorem} means that if we enlarge the switch region in a very careful manner, the set of points that never map into the switch region are precisely those whose codings avoid a finite set of forbidden words. In which case the set of codings of univoque points is a subshift of finite type.

\section*{Acknowledgements}
 The first author was supported by the Dutch Organisation for Scientifitic Research(NWO) grant number 613.001.002. The third author was supported by China Scholarship Council grant number 20123013.

\bibliographystyle{plain}
\bibliography{Uniquecoding.bib}

\end{document}